\documentclass{amsart}

\usepackage{amssymb,amsfonts, amsmath, amsthm}
\usepackage{xypic}
\usepackage[active]{srcltx}

\newcommand{\R}{\Bbb{R}}

\newcommand{\N}{\Bbb{N}}

\newcommand{\rt}{\rtimes}

\newcommand{\eps}{\varepsilon}

\newcommand{\iar}{\ar@{^{(}->}}

\newcommand{\Homeo}{\textup{Homeo}(\Omega)}
\newcommand{\HomN}{\textup{Homeo}(\N)}
\newcommand{\bHomeo}{\textup{Homeo}(\beta \Omega)}

\newcommand{\AutC}{\textup{Aut}^+_s(C_0(\Omega))}

\newcommand{\Aut}{\textup{Aut}^+}
\newcommand{\Auts}{\textup{Aut}^+_s}
\newcommand{\IAut}{\textup{IAut}^+}
\newcommand{\IAuts}{\textup{IAut}^+_s}
\newcommand{\ZAut}{\textup{ZAut}^+}
\newcommand{\ZAuts}{\textup{ZAut}^+_s}

\newcommand{\norm}[1]{\left\Vert #1 \right\Vert}
\newcommand{\comment}[1]{}
\newcommand{\mpos}{(\ell^\infty)^{++}}

\newcommand{\cbo}{C_b(\Omega)^{++}}
\newcommand{\coo}{C_0(\Omega)}

\newcommand{\sdp}{\ZAut(E) \rt \IAut(E)}
\newcommand{\gp}{\ZAut(E) \cdot \IAut(E)}
\newcommand{\gps}{\ZAuts(E) \cdot \IAuts(E)}

\newcommand{\dpsi}{\;d\psi}
\newcommand{\id}{\textup{id}}

\theoremstyle{plain}

\newtheorem{theorem}{Theorem}[section]
\newtheorem{lemma}[theorem]{Lemma}
\newtheorem{prop}[theorem]{Proposition}
\newtheorem{corol}[theorem]{Corollary}

\theoremstyle{definition}

\newtheorem{defn}[theorem]{Definition}
\newtheorem{assumption}[theorem]{Assumption}
\newtheorem{example}[theorem]{Example}

\theoremstyle{remark}

\numberwithin{equation}{section}

\begin{document}

\title[Compact groups of positive operators]
{{Compact groups of positive operators on Banach lattices}}
\author{Marcel de Jeu}
\email{mdejeu@math.leidenuniv.nl}
\address{Mathematical Institute,
Leiden University,
P.O. Box 9512,
2300 RA Leiden,
The Netherlands}
\author{Marten Wortel}
\email{marten.wortel@gmail.com}
\address{Mathematical Institute,
Leiden University,
P.O. Box 9512,
2300 RA Leiden,
The Netherlands}
\subjclass[2010]{Primary 22D12; Secondary 22C05, 46B42}
\keywords{Compact group, positive representation, lattice automorphism, isometric representation, Banach lattice}

\begin{abstract}

In this paper we study groups of positive operators on Banach lattices. If a certain factorization property holds for the elements of such a group, the group has a homomorphic image in the isometric positive operators which has the same invariant ideals as the original group. If the group is compact in the strong operator topology, it equals a group of isometric positive operators conjugated by a single central lattice automorphism, provided an additional technical assumption is satisfied, for which we have only examples. We obtain a characterization of positive representations of a group with compact image in the strong operator topology, and use this for normalized symmetric Banach sequence spaces to prove an ordered version of the decomposition theorem for unitary representations of compact groups. Applications concerning spaces of continuous functions are also considered.

\end{abstract}

\maketitle

\section{Introduction and overview}

In this paper we continue our efforts, initiated in \cite{finitegroups}, to develop a theory of strongly continuous positive representations of locally compact groups in Banach lattices. In \cite{finitegroups} we investigated positive representations of finite groups. We showed that a principal band irreducible positive representation of a finite group in a Riesz space is finite dimensional, and that the representation space is necessarily Archimedean. Furthermore, we classified such irreducible representation and showed that each Archimedean finite dimensional positive representations is an order direct sum of irreducible positive representations.  Here we go a step further and consider strongly continuous positive representations of compact groups in Banach lattices. Given the relative ease with which unitary representations of compact groups can be treated, this is the natural step to take and one would like to achieve a better understanding of issues related to irreducibility and decomposition in this context. Since the image of such a representation is a group of positive operators, we examine groups of positive operators, and since the image is compact in the strong operator topology, we are especially interested in compact (in the strong operator topology) groups of positive operators. In fact, most of the work in this paper is aimed at a better understanding of such compact groups. Once this is achieved, the transition to strongly continuous positive representations with compact image is not complicated.

There are not too many papers on groups of positive operators. In \cite{dyephillips}, uniformly bounded groups of positive operators on $C_c(\Omega)$ and $C_0(\Omega)$ are investigated in detail, where $\Omega$ is a locally compact Hausdorff space. These groups are studied using group actions on the underlying space $\Omega$ and group cohomology methods. Amongst others, it is shown in \cite[Example~4.1]{dyephillips} that a strongly continuous positive representation of a compact group on $C_0(\Omega)$ equals an isometric strongly continuous representation conjugated by a single central lattice automorphism, a result which we obtain as a special case of a more general statement; cf.\ Theorem~\ref{t:characterization_rep_into_E} below. 

In the case where the group $G$ is compact in the strong operator topology, which is the main focus of our paper, a basic  result is \cite[Theorem~III.10.4]{schaefer}. It was published in \cite{nagelwolff}, which in turn is based on unpublished lecture notes by H.P.\ Lotz. It gives information concerning the structure of $G$ as well as the Banach lattice $G$ acts on, under the additional assumption that the action has only trivial invariant closed ideals. Amongst others, it states that the pertinent lattice can be found between $C(G/H)$ and $L^1(G/H)$, for some closed subgroup $H$ of $G$, and the group acts as the group of left quasi-rotations induced by the natural action of $G$ on $G/H$. 

By studying the spectrum of lattice homomorphisms, \cite{schaeferwolffarendt} also contains some results about groups of positive operators, in particular it is shown in \cite[Corollary~3.10]{schaeferwolffarendt} that a uniformly bounded group of positive operators on a Banach lattice is discrete in the norm topology, a result we obtain in the special case of groups which are compact in the strong operator topology on certain sequence spaces, cf.\ Corollary~\ref{c:norm_discrete_seq_space} below.

Beyond these results not much seems to be known. Naturally, there is a theory of one-parameter semigroups of positive operators, see, e.g., \cite{semigroupsposoperators}, but we are not aware of issues of irreducibility or decomposition into irreducibles being considered in detail for such semigroups.

In this paper we study groups of positive operators, or equivalently, groups of lattice automorphisms, with the property that every element can be written as a product of a central lattice automorphism and an isometric lattice automorphism. Remarkably enough, in the Banach lattices we consider in this paper, \emph{every} lattice automorphism is such a product. However, there are Banach lattices for which this fails, cf.\ Example \ref{e:automorphisms_is_not_product}. The Banach lattices for which this holds true, as shown in this paper, include the normalized symmetric Banach sequence spaces (Section~\ref{s:sequence_spaces}) and spaces of continuous functions (Section~\ref{s:cont_functions}). Moreover, in these spaces we have a concrete description of both the central lattice automorphisms and the isometric lattice automorphisms. In the general case, for all groups with the aforementioned factorization property, we show that there is a group of isometric lattice automorphisms with the same invariant ideals as the original group, cf.\ Theorem~\ref{t:group_isom_invariant_bands}. This is applied to the Banach sequence spaces mentioned above, where the isometric lattice automorphisms are easily described as permutations operators, and without too much effort one thus obtains a decomposition of a positive representation of an arbitrary group in such a Banach sequence space into band irreducibles, cf.\ Theorem~\ref{t:compsplitirreducible}. This result is reminiscent of the familiar decomposition theorem for strongly continuous unitary representations of compact groups into finite dimensional irreducible representations, but here the representation need not be strongly continuous, the group need not be compact, and the (order) irreducibles can be infinite dimensional.

Suppose the original group of automorphisms with the above factorization property is compact in the strong operator topology. As a first thought, we can equip the Banach lattice $E$ with an equivalent lattice norm $||| \cdot |||$, defined by 
\[ |||x||| := \int_G \norm{Tx} \,dT \quad \forall x \in E, \]
where $dT$ denotes the Haar measure on the compact group $G$. Under this norm, the group $G$ is now easily seen to be a group of isometric lattice automorphisms. However, by changing the norm, the isometries change as well, and any nice description of the original isometries need not survive this transformation. Hence this does not seem useful. Instead, we impose an additional technical assumption (Assumption~\ref{a:a1}) on the Banach lattice. Under this assumption, which we show to hold for normalized symmetric Banach sequence spaces with order continuous norm and spaces of continuous function, we can actually show that such a compact group is isomorphic as a topological group with the aforementioned group of isometric lattice automorphisms with the same invariant ideals. Moreover, we can characterize such groups $G$: they are precisely the groups of the form $G = mHm^{-1}$, for a unique compact group $H$ of isometric lattice automorphisms, and a (non-unique) central lattice automorphism $m$, cf.\ Theorem~\ref{t:characterization_mult_subgroups}. This is especially useful whenever we have a nice description of the central lattice automorphisms and the isometric lattice automorphisms, as in the spaces mentioned above. Along the same lines, one can show that positive representations with compact image in such spaces are precisely the conjugates of isometric representations, cf.\ Theorem~\ref{t:characterization_rep_into_E}. Moreover, in the case that we have a positive representation in a normalized symmetric Banach sequence space with order continuous norm or in $\ell^\infty$ as in Theorem~\ref{t:compsplitirreducible}, and the positive representation has compact image, the irreducible bands are finite dimensional, so that the analogy with unitary representations of compact groups is then complete. For positive representations with compact image in spaces of continuous function, one cannot in general obtain such a direct sum type decomposition as in Theorem~\ref{t:compsplitirreducible}, and further research is necessary to see whether there is still a structure theorem for such representations in terms of band irreducible ones. As a preparation, we include a number of results on the invariant closed ideals, bands and projection bands for these representations.

\comment{

\bigskip
As mentioned above, in the spaces we consider it is always true that an arbitrary lattice automorphism is the product of a central and an isometric lattice automorphism. As will become clear in Section~\ref{s:abstract_characterization}, this is equivalent to saying that the automorphism group of the lattice is the semidirect product of the group of central lattice automorphisms and the group of isometric lattice automorphisms. The technical Assumption~\ref{a:a1} also holds in all cases under consideration, and as this paper shows, the combination of these properties certainly has its consequences for positive group representations in these spaces. It is too early to speak of conjectures, but it is an intriguing and relevant question whether there are counterexamples to the statement that both these properties actually hold in all Banach lattices.
}

\bigskip

The structure of this paper is as follows.

In Section~\ref{s:preliminaries} we introduce the basic notation and terminology. After establishing a few facts on groups of invertible operators and representations, we give a new proof of the fact that the center of a Banach lattice is isometrically algebra and lattice isomorphic to $C(K)$, for some compact Hausdorff space $K$. We also obtain some results on integrating strongly continuous center valued functions. In Section~\ref{s:abstract_characterization} we consider groups of lattice automorphisms for which every element is the product of a central lattice automorphism and an isometric lattice automorphism. We immediately obtain that there exists a group of isometric lattice automorphisms having the same invariant ideals as the original group. Then we state the technical Assumption~\ref{a:a1}, and under this assumption we are able to show one of our main results, Theorem~\ref{t:characterization_mult_subgroups}, which states that every group of lattice automorphisms with this factorization property, and which is compact in the strong operator topology, equals a group of isometric lattice automorphisms conjugated by a central lattice automorphism. Using similar ideas, it is shown in Section~\ref{s:abstract_representations} that positive representations with compact (in the strong operator topology) image are isometric positive representations conjugated with a central lattice automorphism. We then show that two positive representations with compact image are order equivalent if and only if their isometric parts are (isometrically) order equivalent. In Section~\ref{s:sequence_spaces} we define and examine normalized symmetric Banach sequence spaces. We show that all lattice automorphisms on such spaces are a product of a central lattice automorphism and an isometric lattice automorphism, and that, if the space has order continuous norm, the technical Assumption~\ref{a:a1} holds. Then we apply the results from Section~\ref{s:abstract_characterization} and Section~\ref{s:abstract_representations} to characterize compact groups of lattice automorphisms and positive representations with compact image. We also obtain Theorem~\ref{t:compsplitirreducible}, the aforementioned ordered version of the decomposition theorem for unitary representations of compact groups. Finally, in Section~\ref{s:cont_functions}, we examine the Banach lattice $C_0(\Omega)$ for locally compact Hausdorff spaces $\Omega$. Again we show that all lattice automorphisms are a product of a central lattice automorphism and an isometric lattice automorphism, and that Assumption~\ref{a:a1} holds, and we apply the results from Section~\ref{s:abstract_characterization} and Section~\ref{s:abstract_representations} to characterize compact groups of lattice automorphisms and positive representations with compact image. We finish with Proposition~\ref{p:invariance_coo}, which characterizes invariant closed ideals, bands and projection bands of positive representations with compact image.

\section{Preliminaries}\label{s:preliminaries}

In this section we discuss various facts concerning the strong operator topology, groups of invertible operators, positive representations, the center of a Banach lattice, and integration of strongly continuous center valued functions.

If $X$ is a Banach space, then $\mathcal{L}(X)$ denotes the bounded operators on $X$, and this space equipped with the strong operator topology will be denoted by $\mathcal{L}_s(X)$. Subsets of $\mathcal{L}_s(X)$ are always assumed to be equipped with the strong operator topology. It follows from the principle of uniform boundedness that compact subsets of $\mathcal{L}_s(X)$ are uniformly bounded. In this topology multiplication is separately continuous, and the multiplication is simultaneously continuous when the first variable is restricted to uniformly bounded subsets. The next lemma is concerned with the continuity of the inverse.

\begin{lemma}\label{l:group_is_top_group}
 Let $X$ be a Banach space and $H \subset \mathcal{L}_s(X)$ be a set of invertible operators such that $H^{-1}$ is uniformly bounded. Then taking the inverse in $H$ is continuous.
\end{lemma}

\begin{proof}
 Let $M > 0$ satisfy $\norm{T^{-1}} \leq M$ for all $T \in H$, and let $(T_i)$ be a net in $H$ that converges strongly to $T \in H$. Let $x \in X$, then $x = Ty$ for some $y \in X$, and by the strong convergence of $T_i$ to $T$,
\[ \norm{T_i^{-1}x - T^{-1}x} = \norm{T_i^{-1} (Ty - T_i y)} \leq M \norm{Ty - T_i y} \to 0. \]
\end{proof}

\begin{corol}\label{c:group_is_top_group}
 Let $X$ be a Banach space and let $H \subset \mathcal{L}_s(X)$ be a compact set and a group of invertible operators. Then $H$ is a compact topological group.
\end{corol}

\begin{proof}
 Compact subsets of $\mathcal{L}_s(X)$ are uniformly bounded, so the corollary follows from Lemma \ref{l:group_is_top_group} and the simultaneous continuity of multiplication on uniformly bounded subsets of $\mathcal{L}_s(X)$.
\end{proof}

As a consequence, the group $H$ in Corollary~\ref{c:group_is_top_group} has an invariant measure, a fact which will be instrumental in the proof of the key Lemma~\ref{l:vectint} below.

We continue with another lemma involving the strong operator topology, to be used in Lemma \ref{l:concon}.

\begin{lemma}\label{l:projections_cont}
 Let $X$ be a Banach space, and let $A, B \subset \mathcal{L}_s(X)$ be equipped with the strong operator topology, such that $T_1 S_1 = T_2 S_2$ if and only if $T_1 = T_2$ and $S_1 = S_2$, for all $T_1, T_2 \in A$ and $S_1, S_2 \in B$. Define $p_A \colon A \cdot B \to A$ and $p_B \colon A \cdot B \to B$ by $p_A(TS) := T$ and $p_B(TS) = S$, for $TS \in A \cdot B$. Let $C \subset A \cdot B$ be a subset such that $p_A(C)$ is uniformly bounded and all elements of $p_B(C)$ are surjective. Then, if $p_B$ restricted to $C$ is continuous, $p_A$ restricted to $C$ is continuous as well.
\end{lemma}

\begin{proof}
 Let $M > 0$ satisfy $\norm{T} \leq M$ for all $T \in A$ and $S \in B$ with $TS \in C$. Suppose $p_B$ restricted to $C$ is continuous, and let $(T_i S_i)$ be a net in $C$ that converges strongly to $TS \in C$, where $(T_i)$ is a net in $A$ and $T \in A$, and $(S_i)$ is a net in $B$ and $S \in B$. Let $x \in X$, then $x = Sy$ for some $y \in X$, and
\begin{align*}
 \norm{T_i x - Tx} &= \norm{T_i Sy - TSy} \\
&\leq \norm{T_i Sy - T_i S_i y} + \norm{T_i S_i y - TSy} \\
&\leq M \norm{Sy - S_i y} + \norm{T_i S_i y - TSy},
\end{align*}
which converges to zero by the continuity of $p_B$ and the strong convergence of $(T_i S_i)$ to $TS$.
\end{proof}

Let $E$ be a (real) Banach lattice. Being regular operators on a Banach lattice, lattice automorphisms of $E$ are automatically bounded, and the group of lattice automorphisms of $E$ equipped with the strong operator topology will be denoted by $\Aut(E)$. The subgroup of isometric lattice automorphisms is denoted by $\IAut(E)$. Equipped with the strong operator topology, we will denote these spaces by $\Auts(E)$ and $\IAuts(E)$, and subsets of $\Auts(E)$ and $\IAuts(E)$ are always assumed to have the strong operator topology.

\begin{defn}
 Let $G$ be a group and $E$ a Banach lattice. A \emph{positive representation} of $G$ in $E$ is a group homomorphism $\rho \colon G \to \Aut(E)$.
\end{defn}
For typographical reasons, we will write $\rho_s$ instead of $\rho(s)$, for $s \in G$.

Suppose $\rho \colon G \to \Aut(E)$ and $\theta \colon G \to \Aut(F)$ are positive representations in the Banach lattices $E$ and $F$. A positive operator $T \colon E \to F$ is called a \emph{positive intertwiner} of $\rho$ and $\theta$ if $T \rho_s = \theta_s T$ for all $s \in G$, and $\rho$ and $\theta$ are called \emph{order equivalent} if there exists a positive intertwiner of $\rho$ and $\theta$ which is a lattice automorphism. We call them isometrically order equivalent if there exists an intertwiner in $\IAut(E)$.

We call a positive representation $\rho$ of $G$ in $E$ \emph{band irreducible} if the only $\rho$-invariant bands are $\{0\}$ and $E$. Projection band irreducibility, closed ideal irreducibility, etc., are defined similarly.

In this paper we are, amongst others, concerned with subgroups of $\Auts(E)$. By the above, $\Auts(E)$ is a group with a topology such that the multiplication is separately continuous. We present a useful lemma about such groups, which can also be found in \cite[Lemma~2.4]{crossedproducts}.

\begin{lemma}\label{l:septopgrp}
 Let $G$ and $H$ be two groups with a topology such that right multiplication is continuous in both groups, or such that left multiplication is continuous in both groups. Let $\phi \colon G \to H$ be a homomorphism. Then $\phi$ is continuous if and only if it is continuous at $e$.
\end{lemma}

\begin{proof}
Assume that right multiplication is continuous in both groups. Let $\phi$ be a homomorphism which is continuous at $e$ and let $(r_i)$ be a net in $G$  converging to $r \in G$. Then $r_i r^{-1} \to e$ by the continuity of right multiplication by $r^{-1}$ in $G$, and so
\[ \phi(r_i) = \phi(r_i r^{-1}) \phi(r) \to \phi(r), \]
where the continuity of right multiplication by $\phi(r)$ in $H$ is used in the last step. The case of continuous left multiplication is proved similarly, writing $\phi(r_i) = \phi(r) \phi(r^{-1}r_i)$.
\end{proof}

We continue by examining the center $Z(E)$ of a Banach lattice $E$, which, as in \cite[Definition~3.1.1]{meyernieberg}, is defined to be the set of regular operators $m$ on $E$ satisfying $-\lambda I \leq m \leq \lambda I$ for some $\lambda \geq 0$. With $Z_s(E)$ we denote $Z(E)$ with the strong operator topology. Central operators are often multiplication operators in concrete examples, e.g., if $1 \leq p \leq \infty$ and $(\Sigma, \mu)$ is a finite measure space, then each central operators $m$ on $L^p(\mu)$ is a multiplication operator by an element of $L^\infty(\mu)$, cf.\ the example following \cite[Definition~3.1.1]{meyernieberg}, and this is why we use the notation $m$ for these operators. The center of a Banach lattice is in all respects isomorphic to a space of continuous function, which is the context of the next proposition. For its proof and that of Corollary \ref{c:pos_center_determines_center}, we recall some terminology. If $E$ is a Banach lattice, then Orth$(E)$ denotes the orthomorphisms of $E$, i.e, the order bounded band preserving operators (\cite[Definition~2.41]{posoperators}). An $f$-algebra is a Riesz space $E$ equipped with a multiplication turning $E$ into an associative algebra, such that if $x, y \in E^+$, then $xy \in E^+$, and if $x \wedge y = 0$, then $xz \wedge y = zx \wedge y = 0$ for all $z \in E^+$ (\cite[Definition~2.53]{posoperators}).

\begin{prop}\label{p:center_is_cont_functions}
 Let $E$ be a Banach lattice. Then the center $Z(E)$ equipped with the operator norm is isometrically lattice and algebra isomorphic to the space $C(K)$ for some compact Hausdorff space $K$, such that the identity operator $I$ is identified with the constant function $\mathbf{1}$.
\end{prop}

\begin{proof}
 By \cite[Theorem~3.1.11]{meyernieberg}, the operator norm of $m \in Z(E)$ is the least $\lambda \geq 0$ such that $-\lambda I \leq m \leq \lambda I$, i.e., it equals the order unit norm corresponding to the order unit $I$, and so by \cite[Proposition~1.2.13]{meyernieberg} $Z(E)$ is an $M$-space with order unit $I$. Then the well-known Kakutani Theorem (\cite[Theorem~2.1.3]{meyernieberg}) yields an isometric lattice isomorphism of $Z(E)$ with a $C(K)$ space such that $I$ corresponds to $\mathbf{1}$. Moreover, by \cite[Theorem~3.1.12(ii)]{meyernieberg} $Z(E) = \mbox{Orth}(E)$, which is an Archimedean $f$-algebra by \cite[Theorem~2.59]{posoperators}. Clearly $C(K)$ is an Archimedean $f$-algebra with unit $\mathbf{1}$. By \cite[Theorem~2.58]{posoperators} the $f$-algebra structure on an Archimedean $f$-algebra is unique, given the positive multiplicative unit, and this implies that the correspondence between $Z(E)$ and $C(K)$ must be an algebra isomorphism.
\end{proof}

This proposition is stated in \cite[Proposition~1.4]{schaeferwolffarendt}, where a reference to \cite{hackenbroch} is given for the proof. The development of the theory since the appearance of \cite{hackenbroch} enables us to give a proof as above.

If $E$ is a Banach lattice, then $\ZAut(E)$ denotes the set of central lattice automorphisms, i.e., $\ZAut(E) = \Aut(E) \cap Z(E)$. Note that $\ZAut(E)$ does not denote the center (in the sense of groups) of $\Aut(E)$! As before, $\ZAuts(E)$ denotes $\ZAut(E)$ equipped with the strong operator topology. In the following corollary we collect a few properties of $\ZAut(E)$ as they follow from the isomorphism in Proposition \ref{p:center_is_cont_functions}. If $K$ is a compact Hausdorff space, then $C(K)^{++}$ denotes the strictly positive functions in $C(K)$, or equivalently, the positive multiplicatively invertible elements of $C(K)$.

\begin{corol}\label{c:pos_center_determines_center}
 Let $E$ be a Banach lattice. Under the isomorphism $Z(E) \cong C(K)$ from Proposition \ref{p:center_is_cont_functions}, we have $\ZAut(E) \cong C(K)^{++}$. Consequently, $\ZAut(E)$ is a group, and
\[ Z(E) = \ZAut(E) - \R^+ \cdot I = \ZAut(E) - \ZAut(E). \]
\end{corol}

\begin{proof}
 Suppose $m$ corresponds to an element of $C(K)^{++}$. Then $m^{-1}$ corresponds to an element of $C(K)^{++}$ as well, and so $m$ is positive with a positive inverse and hence a lattice automorphism, i.e., $m \in \ZAut(E)$. Conversely, let $m \in \ZAut(E)$. Then $m$ corresponds to a positive function in $C(K)$. Since $Z(E) = \mbox{Orth}(E)$, \cite[Theorem~3.1.10]{meyernieberg} shows that, if $m \in Z(E)$ is invertible in $\mathcal{L}(E)$, its inverse is in $Z(E)$ as well. So $m^{-1} \in Z(E) \cong C(K)$, which is only possible if $m$ corresponds to an element of $C(K)^{++}$. The final statement now follows from $C(K) = C(K)^{++} - \R^+ \cdot \mathbf{1}$.
\end{proof}

The next lemma yields an isometric action of the group of lattice automorphisms on the center of a Banach lattice.

\begin{lemma}\label{l:conjugation_action}
 Let $E$ be a Banach lattice. Conjugation by elements of $\Aut(E)$ induces a group homomorphism from $\Aut(E)$ into the group of isometric algebra and lattice automorphisms of $Z(E)$. If $H \subset \Auts(E)$ is a uniformly bounded set such that $H^{-1}$ is also uniformly bounded and $A \subset Z_s(E)$ is uniformly bounded, then the map $H \times A \to Z_s(E)$ defined by $(T,m) \mapsto TmT^{-1}$ is continuous. Moreover, if $T \in \Aut(E)$ is fixed, then $m \mapsto TmT^{-1}$ is a continuous algebra and lattice automorphism of $Z_s(E)$.
\end{lemma}

\begin{proof}
Let $T \in \Aut(E)$ and $m \in Z(E)$, and take $\lambda \geq 0$. Then
\begin{align*}
-\lambda x \leq mx \leq \lambda x \quad \forall x \in E^+ &\Leftrightarrow -\lambda T^{-1}y \leq mT^{-1}y \leq \lambda T^{-1}y \quad \forall y \in E^+ \\
&\Leftrightarrow -\lambda y \leq TmT^{-1}y \leq \lambda y \quad \forall y \in E^+,
\end{align*}
hence conjugation by elements of $\Aut(E)$ maps $Z(E)$ isometrically into itself. The conjugation action is obviously an algebra automorphism, and if $m$ is positive, then $TmT^{-1}$ is positive as well, so conjugation is positive with a positive inverse, hence a lattice automorphism. The second statement follows from Lemma \ref{l:group_is_top_group}, and the continuity of $m \mapsto TmT^{-1}$ follows from the separate continuity of multiplication in the strong operator topology.
\end{proof}

Finally, we need a proposition for weak integration of strongly continuous center valued functions. If $X$ is a Banach space, $(H, dh)$ a compact Hausdorff probability space, with which we mean a compact Hausdorff space equipped with a not necessarily regular Borel probability measure, and $g \colon H \to X$ a continuous function, then \cite[Theorem~3.27]{rudin} shows that there exists a unique element of $X$, denoted by $\int_H g(h) \;dh$, defined by duality as follows:
\begin{align}\label{e:vect_int}
 \left\langle \int_H g(h) \;dh, x^* \right\rangle = \int_H \langle g(h), x^* \rangle \;dh \quad \forall x^* \in X^*.
\end{align}
Moreover, $\int_H g(h) \;dh \in \overline{\mbox{co}}(g(H))$. By applying functionals it easily follows that bounded operators can be pulled through the integral, and that the triangle inequality holds.

The above vector valued integral will be used in the next proposition to define an operator valued integral. The Banach space part of the next proposition is a standard argument, which we repeat here for the convenience of the reader.

\begin{prop}\label{p:central_valued_integration}
 Let $(H, dh)$ be a compact Hausdorff probability space, $E$ a Banach space and $f \colon H \to \mathcal{L}_s(E)$ a continuous map. Then the operator $\int_H f(h) \;dh \colon E \to E$, defined by
\[ \left( \int_H f(h) \;dh \right) x := \int_H f(h)x \;dh \quad \forall x \in E, \]
where the second integral is defined by \eqref{e:vect_int}, defines an element of $\mathcal{L}(E)$ satisfying $\norm{\int_h f(h) \;dh} \leq \sup_{h \in H} \norm{f(h)}$. If $S,T \in \mathcal{L}(E)$, then
\begin{equation}\label{e:propintegral}
 S \left( \int_H f(h) \;dh \right) T = \int_H S f(h) T \;dh.
\end{equation}
Moreover, if $E$ is a Banach lattice and $f(H) \subset Z(E)$, then there exist $\lambda, \mu \in \R$ such that $f(H) \subset [\lambda I, \mu I]$, and for such $\lambda$ and $\mu$ we have $\int_H f(h) \;dh \in [\lambda I, \mu I] \subset Z(E)$.
\end{prop}

\begin{proof}
Note that $f(H)$ is uniformly bounded by the principle of uniform boundedness. The computation
\[ \norm{\left( \int_H f(h) \;dh \right) x} = \norm{\int_H f(h)x \;dh} \leq \int_H \norm{f(h)x} \;dh \leq \sup_{h \in H} \norm{f(h)} \norm{x} \]
shows that the linear operator $\int_H f(h) \;dh$ is bounded and that its norm satisfies the required estimate. By applying elements of $E$ and functionals, and using the properties of the $E$-valued integral, \eqref{e:propintegral} easily follows.

Now assume $E$ is a Banach lattice and $f(H) \subset Z(E)$. By the uniform boundedness of $f(H)$, there exist $\lambda, \mu \in \R$ such that $f(H) \subset [\lambda I, \mu I]$. Suppose $\lambda$ and $\mu$ satisfy this relation, then we have to show that $(\int_H f(h) \;dh) x \in [\lambda x, \mu x]$ for all $x \in E^+$, which is equivalent with
\[ \int_H \lambda x \;dh \leq \int_H f(h)x \;dh \leq \int_H \mu x \;dh. \]
Now $f(h)x - \lambda x \in E^+$ for all $h \in H$ by assumption, and since $E^+$ is a closed convex set, the properties of the $E$-valued integral imply that $\int_H [f(h)x - \lambda x] \;dh \in E^+$ as well. The second inequality follows similarly.
\end{proof}

\section{Groups of positive operators}\label{s:abstract_characterization}

In this section we will relate certain groups of lattice automorphisms to groups of isometric lattice automorphisms. The main assumption on these groups is that every element in the group can be written as a product of a central lattice automorphism and an isometric lattice automorphism. Examples of Banach lattices where this assumption is always satisfied are normalized symmetric Banach sequence spaces, such as $c_0$ and $\ell^p$ for $1 \leq p \leq \infty$, where the fact that lattice automorphisms map atoms to atoms easily implies the above property, cf.\ Section~\ref{s:sequence_spaces}, and spaces of continuous functions, where there is a well-known characterization of lattice homomorphisms in terms of a multiplication operator and an operator arising from a homeomorphism of the underlying space, cf.\ Section~\ref{s:cont_functions}. When this assumption is satisfied, we are able to show that there exists a group of isometric lattice automorphisms which has the same invariant ideals as the original group, cf.\ Theorem~\ref{t:group_isom_invariant_bands}. The main result, Theorem~\ref{t:characterization_mult_subgroups}, shows that, under the technical Assumption~\ref{a:a1}, for every compact group $G$ of lattice automorphisms in which every element can be written as a product of a central lattice automorphism and an isometric lattice automorphisms, there exist a unique compact group $H$ of isometric lattice automorphisms and a non-unique central lattice automorphism $m$ such that $G = mHm^{-1}$.

We start by showing that a certain set of lattice automorphisms is actually a group, and, in fact, a non-trivial semidirect product. Recall that the group of isometric lattice automorphisms of a Banach lattice $E$ is denoted by $\IAut(E)$, that the group of central lattice automorphisms of $E$ is denoted by $\ZAut(E)$, and that equipped with the strong operator topology these spaces are denoted by $\IAuts(E)$ and $\ZAuts(E)$. The space $\gp$ equipped with the strong operator topology is denoted by $\gps$.

Obviously, $\IAut(E)$ is a group, and, although not quite so obvious, $\ZAut(E)$ is also a group by Corollary~\ref{c:pos_center_determines_center}. Now suppose $m \phi \in \gp$, then $(m \phi)^{-1} = \phi^{-1} m^{-1} = (\phi^{-1} m^{-1} \phi) \phi^{-1}$, which is in $\gp$ by Lemma~\ref{l:conjugation_action}. In a similar vein, if $m_1 \phi_1, m_2 \phi_2 \in \gp$, then $m_1 \phi_1 m_2 \phi_2 = m_1 (\phi_1 m_2 \phi_1^{-1}) \phi_1 \phi_2 \in \gp$. Hence $\gp$ is a subgroup of $\Aut(E)$.

Moreover, the representation of an element $m \phi \in \gp$ is unique, and to show this it is sufficient to show that $\ZAut(E) \cap \IAut(E) = \{I\}$. So suppose $m \in \ZAut(E)$ is an isometry. Then $\norm{m} = \norm{m^{-1}} = 1$, and taking into account the isometric isomorphism of $Z(E)$ with a $C(K)$ space of Proposition~\ref{p:center_is_cont_functions}, the continuous function corresponding to $m$ must be unimodular. Since this function is also positive, it must be identically one, and so $m = I$.

By Lemma \ref{l:conjugation_action} the group $\IAut(E)$ acts on $\ZAut(E)$ by conjugation, and for $\phi \in \IAut(E)$ and $m \in \ZAut(E)$, this action will be denoted by $\phi(m)$, so $\phi(m) = \phi m \phi^{-1}$. We can form the semidirect product $\sdp$, with group operation
\[ (m_1, \phi_1) (m_2, \phi_2) := (m_1 \phi_1(m_2), \phi_1 \phi_2). \]
Using that $\phi(m)$ is the conjugation action of $\phi \in \IAut(E)$ on $m \in \ZAut(E)$, it easily follows that the map $\chi \colon \sdp \to \gp$ defined by $\chi(m, \phi) := m \phi$ is a group isomorphism.

All in all, it is now clear that $\gp$ is a subgroup of $\Aut(E)$, and that it is isomorphic with $\sdp$. If necessary we identify $\sdp$ and $\gp$ through $\chi$. The map $p \colon \gp \to \IAut(E)$ defined by $p(m \phi) := \phi$ is the projection onto the second factor of the semidirect product, which is a group homomorphism.

In the rest of this section we will assume that the group of lattice automorphisms under consideration is contained in $\gp$. For certain sequence spaces and spaces of continuous function, we will show that $\gp$ equals the whole group of lattice automorphisms, cf.\ Section \ref{s:sequence_spaces} and Section \ref{s:cont_functions}, but the next example, which was communicated to us by A.W.\ Wickstead, shows that there is a simple Banach lattice, not every lattice automorphism of which is a product of a central lattice automorphism and an isometric lattice automorphism.

\begin{example}\label{e:automorphisms_is_not_product}
 Consider $\R^2$ with the usual ordering, and norm $\norm{(x,y)} := \max\{ |y| , |x| + |y|/2 \}$, so that it becomes a Banach lattice with the standard unit vectors having norm one. Hence $(x,y) \mapsto (y,x)$ is the only possible nontrivial isometric lattice automorphism, but this map is not isometric since $\norm{(1,2)} = 2$ whereas $\norm{(2,1)} = 5/2$. Therefore $(x,y) \mapsto (y,x)$ cannot be a product of a central lattice automorphism and an isometric lattice automorphism, since it is not a central lattice automorphism.
\end{example}

\begin{theorem}\label{t:group_isom_invariant_bands}
 Let $E$ be a Banach lattice and $G \subset \gp$ a group. Then $p(G)$ is a group of isometric lattice automorphisms, with the same invariant ideals as $G$.
\end{theorem}

\begin{proof}
 Let $m \in Z(E)$ and $x \in E^+$. Then there exists a $\lambda \geq 0$ such that $-\lambda x \leq mx \leq \lambda x$, and so $mx$ is contained in the ideal generated by $x$. This fact extends to all $x \in E$ by writing $x = x^+ - x^-$, and so $m$ leaves all ideals in $E$ invariant.

 Now let $m \phi \in G$, with $m \in \ZAut(E)$ and $\phi \in \IAut(E)$, let $x \in E$ and let $I \subset E$ be an ideal. Since $m, m^{-1} \in Z(E)$, by the above $x \in I$ if and only if $mx \in I$, and so $I$ is invariant for $m \phi$ if and only if $I$ is invariant for $\phi = p(m \phi)$.
\end{proof}

We will now examine groups $G \subset \gps$ which are compact (in the strong operator topology). In this case we can say much more than Theorem~\ref{t:group_isom_invariant_bands}, if the following assumption on the Banach lattice is satisfied.

\begin{assumption}\label{a:a1}
 If $p \colon \gps \to \IAuts(E)$ denotes the group homomorphism $m \phi \mapsto \phi$, then $p|_G$ is continuous for any compact subgroup $G \subset \gps$.
\end{assumption}

The next proposition allows us to associate compact subgroups of $\IAuts(E)$ with compact subgroups of $\gps$.

\begin{prop}\label{p:bijection_compact_subgroups}
 Let $E$ be a Banach lattice satisfying Assumption \ref{a:a1}. Let $A$ be the set of compact subgroups $G \subset \gps$, and let $B$ be the set of pairs $(H, q)$, where $H \subset \IAuts(E)$ is a compact subgroup and $q \colon H \to \gps$ is a continuous homomorphism such that $p \circ q = \id_H$. Define $\alpha \colon A \to B$ and $\beta \colon B \to A$ by
\[ \alpha(G):= (p(G), (p|_G)^{-1}), \quad \beta(H,q):= q(H). \]
Then for each $G \in A$, $G \mapsto p(G)$ is an isomorphism of compact groups, and $\alpha$ and $\beta$ are inverses of each other.
\end{prop}

\begin{proof}
If $G \in A$, then by Assumption \ref{a:a1} $\ker(p|_G)$ is a compact subgroup of $\ZAuts(E)$, which is isometrically isomorphic with the group $C(K)^{++}$ of strictly positive continuous functions on some compact Hausdorff space $K$ by Corollary~\ref{c:pos_center_determines_center}. By the principle of uniform boundedness, compact subgroups of $\ZAuts(E)$ are uniformly bounded, and obviously the only uniformly bounded subgroup of $C(K)^{++}$ is trivial, hence $p|_G$ is a group isomorphism. Moreover it is a continuous bijection between a compact space and a Hausdorff space, hence $(p|_G)^{-1}$ is continuous. Clearly $p \circ (p|_G)^{-1} = \id_{p(G)}$, so $\alpha$ is well defined.

Let $G \in A$, then $\beta(\alpha(G)) = \beta(p(G), (p|_G)^{-1}) = G$. Conversely, let $(H,q) \in B$, then $\alpha(\beta(H, q)) = \alpha(q(H)) = (p(q(H)), (p|_{q(H)})^{-1})$, and since $p \circ q = \id_H$ it follows that $p(q(H)) = H$ and that $(p|_{q(H)})^{-1} = (p|_{q(H)})^{-1} \circ p \circ q = q$.
\end{proof}

By the above proposition the compact subgroups $G \subset \gps$ are parametrized by the pairs $(H,q)$ of compact subgroups $H \subset \IAuts(E)$ and continuous homomorphism $q \colon H \to \gps$ satisfying $p \circ q = \id_H$. We will now investigate such maps $q$, for a given compact subgroup $H$ of $\IAuts(E)$. The condition $p \circ q = \id_H$ is equivalent with the existence of a map $f \colon H \to \ZAut(E)$ such that $q(\phi) = f(\phi)\phi$, for $\phi \in \IAut(E)$. We now describe the relation between the continuity of $f$ and the continuity of $q$.

\begin{lemma}\label{l:concon}
Let $E$ be a Banach lattice satisfying Assumption \ref{a:a1}, let $H \subset \IAuts(E)$ be a compact group and let $q \colon H \to \gps$ be a group homomorphism of the form $q(\phi) := f(\phi)\phi$, for some map $f \colon H \to \ZAuts(E)$. Then $q$ is continuous if and only if $f$ is continuous.
\end{lemma}

\begin{proof}
 Suppose $f$ is continuous. Then $q$ is the composition of $f$ and the identity map with the multiplication in $\Auts(E)$. Since $f(H)$ is compact it is uniformly bounded, and multiplication in the strong operator topology is simultaneously continuous if the first factor is restricted to uniformly bounded sets. Therefore $q$ is continuous.

Conversely, suppose that $q$ is continuous. Then $q(H)$ is a group and it is compact. Moreover $q(H)$ is uniformly bounded, and so $f(H)$, the set of first coordinates of $q(H)$, is also uniformly bounded, since $\norm{f(\phi)} = \norm{f(\phi) \phi} = \norm{q(\phi)}$ for $\phi \in \IAut(E)$ by the fact that $\phi$ is an isometric automorphism. Since the projection onto the second coordinate is continuous on $q(H)$ by Assumption~\ref{a:a1}, Lemma~\ref{l:projections_cont} yields the continuity on $q(H)$ of the projection onto the first coordinate. It follows that $f$ is continuous as a composition of $q$ and the projection of $q(H)$ onto the first coordinate.
\end{proof}

We continue describing the structure of maps $q$ as above. For $\phi, \psi \in H$ we have $q(\phi \psi) = f(\phi \psi)\phi \psi$ and
\[ q(\phi) q(\psi) = f(\phi)\phi f(\psi)\psi = f(\phi) \phi(f(\psi)) \phi \psi.    \]
Hence $q$ being a homomorphism is equivalent with $f(\phi \psi) = f(\phi) \phi(f(\psi))$ for all $\phi, \psi \in H$, and such maps are called crossed homomorphisms. We will first show that the image of such crossed homomorphisms is bounded from below.

\begin{lemma}\label{l:crossed_hom_bounded_below}
 Let $E$ be a Banach lattice, let $H \subset \IAuts(E)$ be a compact group and let $f \colon H \to \ZAuts(E)$ be a continuous crossed homomorphism, i.e., a continuous map such that $f(\phi \psi) = f(\phi) \phi(f(\psi))$ for all $\phi, \psi \in H$. Then there exists an $\eps > 0$ such that $f(\phi) \geq \eps I$ for all $\phi \in H$.
\end{lemma}

\begin{proof}
 Since $f(H)$ is compact and hence uniformly bounded, there exists some $\lambda > 0$ such that, for all $\phi \in \IAut(E)$,
\begin{equation}\label{e:bound_by_lambda}
 \norm{\phi(f(\phi^{-1}))} = \norm{f(\phi^{-1})} \leq \lambda,
\end{equation}
since $\phi$ acts isometrically on $Z(E)$ by Lemma \ref{l:conjugation_action}. We identify $\ZAut(E)$ with $C(K)^{++}$ for some compact Hausdorff space $K$, using Corollary \ref{c:pos_center_determines_center}. Then \eqref{e:bound_by_lambda} implies
\begin{equation}\label{e:pointwise_bound_by_lambda}
 0 < \phi(f(\phi^{-1})) \leq \lambda
\end{equation}
pointwise on $K$.

 By taking $\phi = \psi = I$ in the definition of a crossed homomorphism, we obtain $f(I) = f(I)f(I)$ and so $\mathbf{1} = f(I)$. For arbitrary $\phi \in \IAut(E)$ we obtain
\[ \mathbf{1} = f(I) = f(\phi \phi^{-1}) = f(\phi) \cdot \phi(f(\phi^{-1})), \]
and so $f(\phi) \geq 1 / \lambda$ pointwise on $K$ by \eqref{e:pointwise_bound_by_lambda}, which establishes the lemma with $\eps = 1 / \lambda$.
\end{proof}

To characterize the continuous crossed homomorphisms, we will use the following lemma. It can be viewed as an analytic version of \cite[Lemma~4.2]{finitegroups}, which is a standard argument in group cohomology.

\begin{lemma}\label{l:vectint}
Let $E$ be a Banach lattice, and let $H \subset \IAuts(E)$ be a compact group. Let $f \colon H \to \ZAuts(E)$ be a strongly continuous map. Then $f$ is a continuous crossed homomorphism, i.e., a continuous map such that $f(\phi \psi) =  f(\phi) \phi(f(\psi))$, where $\phi(m)$ denotes the conjugation action of $\phi \in \Aut(E)$ on $m \in \ZAut(E)$, if and only if there exists an $m \in \ZAut(E)$ such that $f(\phi) = m \phi(m)^{-1}$ for all $\phi \in H$.
\end{lemma}

\begin{proof}
Suppose $f$ is a continuous crossed homomorphism. The group $H$ is a compact topological group by Corollary \ref{c:group_is_top_group}, and so we can equip $H$ with its normalized Haar measure $d\psi$. By Proposition~\ref{p:central_valued_integration} there exist $\lambda, \mu \in \R$ such that $f(H) \subset [\lambda I, \mu I]$ and by Lemma \ref{l:crossed_hom_bounded_below} we may assume that $\lambda > 0$. We use Proposition \ref{p:central_valued_integration} to define $m := \int_H f(\psi) \dpsi$ and also to conclude that this integral is in $ [\lambda I, \mu I]$. By Corollary~\ref{c:pos_center_determines_center}, $[\lambda I, \mu I] \subset \ZAut(E)$, and so $m \in \ZAut(E)$. Then, for $\phi \in H$, by the left invariance of $d\psi$ and the fact that bounded operators can be pulled through the integral by \eqref{e:propintegral},
\begin{align*}
 \phi(m) &= \phi \left( \int_H f(\psi) \dpsi \right) \\
&= \int_H \phi(f(\psi)) \dpsi \\
&= \int_H f(\phi)^{-1} f(\phi \psi) \dpsi \\
&= f(\phi)^{-1} \int_H f(\psi) \dpsi \\
&= f(\phi)^{-1} m,
\end{align*}
showing that $f(\phi) = m \phi(m)^{-1}$. Conversely, any $f$ defined as above is continuous by Lemma \ref{l:conjugation_action}, and such an $f$ is easily seen to be a crossed homomorphism.
\end{proof}

Putting everything together yields the following.

\begin{theorem}\label{t:characterization_mult_subgroups}
Let $E$ be a Banach lattice satisfying Assumption \ref{a:a1}, and let $G \subset \gps$ be a compact group. Then there exist a unique compact group $H \subset \IAuts(E)$ and an $m \in \ZAut(E)$ such that
\[ G = mHm^{-1}. \]
Conversely, if $H \subset \IAuts(E)$ is a compact subgroup and $m \in \ZAut(E)$, then $G \subset \gps$ defined by the above equation is a compact subgroup of $\gps$.
\end{theorem}

\begin{proof}
By Proposition~\ref{p:bijection_compact_subgroups}, the compact subgroups of $\gps$ are precisely the groups $q(H)$, where $H$ is a compact subgroup of $\IAuts(E)$ and $q \colon H \to \gps$ is a continuous homomorphism satisfying $p \circ q = \id_H$. As a consequence of Lemma~\ref{l:concon} and Lemma~\ref{l:vectint}, we see that the compact subgroups of $\gps$ are precisely the groups of the form $\{ m \phi(m)^{-1} \phi: \phi \in H \} = mHm^{-1}$, where $H$ is a compact subgroup of $\IAuts(E)$, and $m \in \ZAut(E)$. This establishes the theorem except for the uniqueness of $H$. As to this, if $G = \{m \phi(m)^{-1} \phi: \phi \in H \}$, for a compact group $H \subset \IAuts(E)$ and $m \in \ZAut(E)$, then, in the notation of Proposition~\ref{p:bijection_compact_subgroups}, $G = \beta(H, q)$, where $q(\phi) = m \phi(m)^{-1} \phi$ for $\phi \in H$. Since $(H, q) = \alpha (\beta (H,q)) = \alpha(G)$, this implies that $H = p(G)$. Hence $H$ is unique.
\end{proof}

Note that, given a compact subgroup $G \subset \gps$, the compact subgroup $H \subset \IAuts(E)$ is unique, but the element $m \in \ZAut(E)$ in Theorem~\ref{t:characterization_mult_subgroups} is obviously not unique, e.g., both $m$ and $\lambda m$ for $\lambda > 0$ generate the same $G$.

\section{Positive representations with compact image}\label{s:abstract_representations}

In this section we will apply the results from the previous section, in particular Proposition~\ref{p:bijection_compact_subgroups} and Lemma~\ref{l:vectint}, to representations of groups with compact (in the strong operator topology) image in Banach lattices satisfying Assumption \ref{a:a1}.

\begin{theorem}\label{t:characterization_rep_into_E}
 Let $E$ be a Banach lattice satisfying Assumption \ref{a:a1}, let $G$ be a group and let $\rho \colon G \to \gps$ a positive representation with compact image. Then there exist a unique positive representation $\pi \colon G \to \IAuts(E)$ and an $m \in \ZAut(E)$ such that
\[ \rho_s = m \pi_s m^{-1} \quad \forall s \in G. \]
The image of $\pi$ is compact. Conversely, any positive representation $\pi \colon G \to \IAuts(E)$ with compact image and $m \in \ZAut(E)$ define a positive representation $\rho$ with compact image in $\gps$ by the above equation.

In this correspondence between $\rho$ and $\pi$, $\rho$ is strongly continuous if and only if $\pi$ is strongly continuous.
\end{theorem}

\begin{proof}
 Since $\rho(G)$ is compact, Proposition~\ref{p:bijection_compact_subgroups} applies, and so, combining this with Lemma~\ref{l:vectint}, $p \colon \rho(G) \to p \circ \rho(G)$ has an inverse of the form $q(\phi) = m \phi(m)^{-1} \phi$ for some $m \in \ZAut(E)$ and all $\phi \in p \circ \rho(G)$. We define $\pi := p \circ \rho$, then $\pi$ has compact image, and for $s \in G$,
\[ \rho_s = (q \circ p)(\rho_s) = q(\pi_s) = m \pi_s(m)^{-1} \pi_s = m \pi_s m^{-1}. \]
This shows the existence of $\pi$. The uniqueness of $\pi$ follows from the uniqueness of the factors in $\ZAut(E)$ and $\IAut(E)$ in
\[ \rho_s = m \pi_s m^{-1} = [m \pi_s(m)^{-1}] \pi_s. \]
The remaining statements are now clear.
\end{proof}

Note that, as in Theorem~\ref{t:characterization_mult_subgroups}, $\pi$ is unique, but $m$ is not. Given the positive representation with compact image $\pi$, $m_1$ and $m_2$ induce the same positive representation with compact image if and only if $m_1^{-1} m_2$ commutes with $\pi_s$ for all $s \in G$, i.e., if and only if $m_1^{-1} m_2$ intertwines $\pi$ with itself.

Any representation as in Theorem \ref{t:characterization_rep_into_E} is, by that same theorem, obviously order equivalent to an isometric representation. In fact, we can say more. The next proposition has the same proof as \cite[Proposition~4.6]{finitegroups}.

\begin{prop}\label{p:rep_isomorphic_isometric_rep}
 Let $E$ be a Banach lattice satisfying Assumption \ref{a:a1}, let $G$ be a group and, using Theorem~\ref{t:characterization_rep_into_E}, let $\rho^1 = m_1 \pi^1 m_1^{-1}$ and $\rho^2 = m_2 \pi^2 m_2^{-1}$ be positive representations of $G$ with compact image in $\gps$, where $\pi^1$ and $\pi^2$ are isometric positive representations with compact image in $\IAuts(E)$, and $m^1, m^2 \in \ZAut(E)$. Then $\rho^1$ and $\rho^2$ are order equivalent if and only if $\pi^1$ and $\pi^2$ are isometrically order equivalent.
\end{prop}

\begin{proof}
 In this proof we use semidirect product notation. Suppose that $\rho^1$ and $\rho^2$ are order equivalent and let $T = (m, \phi) \in \Aut(E)$ be a positive intertwiner. Then, for all $s \in G$,
\begin{align}
\rho^1_s T &= ( m_1 \pi^1_s(m_1)^{-1}, \pi^1_s ) (m, \phi) = ( m_1 \pi^1_s(m_1)^{-1} \pi^1_s(m), \pi^1_s \phi ) \label{e:intertwiner1}\\
T \rho^2_s &= (m, \phi) ( m_2 \pi^2_s(m_2)^{-1}, \pi^2_s ) = (m \phi(m_2) \phi (\pi^2_s(m_2)^{-1}) , \phi \pi^2_s), \label{e:intertwiner2}
\end{align}
and since these are equal, $\phi$ is a positive isometric intertwiner between $\pi^1$ and $\pi^2$.

Conversely, let $\phi$ be a positive isometric intertwiner between $\pi^1$ and $\pi^2$. Then, by taking $m = m_1 \phi(m_2)^{-1}$ and $T = (m, \phi) \in \Aut(E)$, it is easily verified that, for all $s \in G$,
\[ (m_1 \pi^1_s(m_1)^{-1} \pi^1_s(m), \pi^1_s \phi ) = (m \phi(m_2) \phi (\pi^2_s(m_2)^{-1}) , \phi \pi^2_s), \]
 and so, by \eqref{e:intertwiner1} and \eqref{e:intertwiner2}, $T$ intertwines $\rho^1$ and $\rho^2$.
\end{proof}

\section{Positive representations in Banach sequence spaces}\label{s:sequence_spaces}

In this section we consider positive representations of groups in certain sequence spaces. First we show that every lattice automorphism can be written as a product of a central lattice automorphism and an isometric lattice automorphism. We are also able to show that a large class of sequence spaces satisfy Assumption~\ref{a:a1}, and an application of the results from Section~\ref{s:abstract_characterization} and Section~\ref{s:abstract_representations} then yields a description of compact groups of lattice automorphisms and of positive representations in these spaces with compact image. Using Theorem~\ref{t:group_isom_invariant_bands}, we obtain a decomposition of positive representations into band irreducibles, cf.\ Theorem~\ref{t:compsplitirreducible}. If the representation has compact image, then the irreducible bands in the decomposition in Theorem~\ref{t:compsplitirreducible} are finite dimensional.

We consider normalized symmetric Banach sequence spaces $E$, by which we mean Banach lattices of sequences equipped with the pointwise ordering and lattice operations such that if $x \in E$ and $y$ is a sequence such that $|y| \leq |x|$, then $y \in E$ and $\norm{y} \leq \norm{x}$, permutations of sequences in $E$ remain in $E$ with the same norm, and the standard unit vectors $\{e_n\}_{n \in \N}$ have norm 1. Important examples are the classical sequence spaces $c_0$ and $\ell^p$ for $1 \leq p \leq \infty$.

If $x$ is a sequence, then $x_{>N}$ denotes the sequence $x$ but with the first $N$ coordinates equal to $0$; similarly, $x_{\leq N}$ denotes the sequence $x$ with the coordinates greater than $N$ equal to $0$.

We will now show that a normalized symmetric Banach sequence space $E$ satisfies $\Aut(E) = \gp$. A lattice automorphism must obviously map positive atoms to positive atoms, so for each $T \in \Aut(E)$ and $n \in \N$, there exist a unique $m \in \N$ and $\lambda_{mn} > 0$ such that $Te_n = \lambda_{mn} e_m$. Since each $x \in E^+$ is the supremum of the $x_{\leq N}$ for $N \in \N$, the linear span of atoms is order dense, and hence the above relation determines $T$ uniquely. Therefore $T$ can be written as the product of an invertible positive multiplication operator and a permutation operator. We identify the group of permutation operators with $S(\N)$, so each $\phi \in S(\N)$ corresponds to the operator defined by $(\phi x)_n := x_{\phi^{-1}(n)}$ for $x \in E$ and $n \in \N$. The multiplication operators are identified with $\ell^\infty$, and by $\mpos$ we denote the set of elements $m \in \ell^\infty$ for which there exists a $\delta > 0$ such that $m_n \geq \delta$ for all $n \in \N$. We conclude that there exist an $m \in \mpos$ and a $\phi \in S(\N)$ such that $T = m \phi$. Conversely, an operator defined in this way is a lattice automorphism. Obviously $T$ is a central lattice automorphism if and only if its permutation part is trivial, and so the central lattice automorphisms equal $\mpos$. By Corollary \ref{c:pos_center_determines_center} the center $Z(E)$ equals $\ell^\infty$. Obviously $S(\N) = \IAut(E)$, and so $\Aut(E) = \mpos \cdot S(\N) = \ZAut(E) \cdot \IAut(E)$.

For $\phi \in S(\N)$ and $m \in \ell^\infty$, define $\phi(m) \in \ell^\infty$ by $\phi(m)_i := m_{\phi^{-1}(i)}$, the sequence $m$ permuted according to $\phi$. Then, for $n \in \N$,
\[ \phi m \phi^{-1} e_n = \phi m e_{\phi^{-1}(n)} = \phi m_{\phi^{-1}(n)} e_{\phi^{-1}(n)} = m_{\phi^{-1}(n)} e_n = \phi(m)_n e_n = \phi(m) e_n, \]
which shows that $m \mapsto \phi(m)$ equals the conjugation action of $\phi$ on $m$.

We will now show that Assumption \ref{a:a1} holds, if $E$ has order continuous norm. We will actually show more, namely that $p \colon \Aut(E) \to S(\N)$ is continuous. The following lemma is a preparation.

\begin{lemma}\label{l:SNtop}
 Suppose $E$ is a normalized symmetric Banach sequence space. Then the strong operator topology on $S(\N) \subset \Auts(E)$ is stronger than the topology of pointwise convergence. If $E$ has order continuous norm, then the topologies are equal.
\end{lemma}

\begin{proof}
If $\phi, \psi, \psi' \in S(\N)$ and $(\phi_i)$ is a net in $S(\N)$ converging pointwise to $\phi$, then $\psi \phi_i \psi' \to \psi \phi \psi'$ pointwise, so multiplication is separately continuous in the topology of pointwise convergence. We show that the identity map from $S(\N)$ equipped with the strong operator topology to $S(\N)$ equipped with the topology of pointwise convergence is continuous, and by Lemma \ref{l:septopgrp} we only have to verify continuity at the identity. Let $(\phi_i)$ be a net in $S(\N)$ converging strongly to the identity, and suppose there is an $n \in \N$ such that $\phi_i(n)$ does not converge to $n$. If $i$ is such that $\phi_i(n) \not= n$, then
$|\phi_i e_n - e_n| \geq e_n$ and so
\[ \norm{\phi_i e_n - e_n} \geq \norm{e_n} = 1. \]
It follows that $\phi_i e_n$ does not converge to $e_n$, which is a contradiction. This shows that $\phi_i$ converges pointwise to the identity.

Now suppose that $E$ has order continuous norm. We will show that the identity map from $S(\N)$ equipped with the topology of pointwise convergence to $S(\N)$ equipped with the strong operator topology is continuous, for which we again only have to verify continuity at the identity. Let $(\phi_i)$ be a net in $S(\N)$ converging pointwise to the identity. Let $x \in E$ and $\eps > 0$. Since $|x|_{>N} \downarrow 0$ for $N \to \infty$, by the order continuity of the norm we can choose $N$ such that $\norm{x_{>N}} = \norm{|x|_{>N}} < \eps / 2$. Choose $j$ such that  $\phi_i$ is the identity on all indices $n \leq N$, for all $i \geq j$. Then, for all $i \geq j$,
\[ \norm{\phi_i x - x} \leq \norm{\phi_i (x_{>N})} + \norm{x_{>N}} < \frac{\eps}{2} + \frac{\eps}{2} = \eps, \]
hence $\phi_i$ converges strongly to the identity.
\end{proof}

This lemma can be used to show that Assumption~\ref{a:a1} holds, if $E$ has order continuous norm.

\begin{lemma}
Let $E$ be a normalized symmetric Banach sequence space with order continuous norm. Then the homomorphism $p \colon \Auts(E) \to S(\N)$ is continuous.
\end{lemma}

\begin{proof}
Again by Lemma~\ref{l:septopgrp} it suffices to show continuity at the identity. So let $(m_i \phi_i)$ be a net in $\Auts(E)$ that converges strongly to the identity, and suppose $\phi_i$ does not converge to the identity. Then by Lemma~\ref{l:SNtop} there is an $n \in \N$ such that $\phi_i(n)$ does not converge to $n$, and, for $i$ such that $\phi_i(n) \not= n$, we obtain
\[ \norm{m_i \phi_i e_n - e_n} = \norm{m_i e_{\phi_i^{-1}(n)} - e_n} \geq \norm{e_n} = 1. \]
This contradicts the assumption that $m_i \phi_i$ converges strongly to the identity, and so $\phi_i$ does converge to the identity.
\end{proof}

Note that $\ell^\infty$ is a normalized Banach sequence space and hence satisfies $\Aut(\ell^\infty) = \gp$. It does not have order continuous norm, but it is isometrically lattice isomorphic to $C(K)$ for some compact Hausdorff space $K$ by Kakutani's Theorem, and in Section~\ref{s:cont_functions} we will show, in Lemma~\ref{l:assumption_verified_for_coo}, that such spaces also satisfy Assumption~\ref{a:a1}.

Now we can apply the theory of the previous sections, in particular Theorem~\ref{t:characterization_mult_subgroups}, Theorem~\ref{t:characterization_rep_into_E} and Proposition~\ref{p:rep_isomorphic_isometric_rep}, to obtain the following characterization of compact subgroups of $\Auts(E)$ and of positive representations with compact image.

\begin{theorem}\label{t:charac_group_in_seqspace}
 Let $E$ be a normalized symmetric Banach sequence space with order continuous norm or $\ell^\infty$, and let $G \subset \Auts(E)$ be a compact group. Then there exist a unique compact group $H \subset S(\N)$ and an $m \in \mpos$ such that
\[ G = mHm^{-1}. \]
Conversely, if $H \subset S(\N)$ is a compact group and $m \in \mpos$, then $G \subset \Aut(E)$ defined by the above equation is a compact subgroup of $\Auts(E)$.
\end{theorem}

\begin{theorem}\label{t:charac_reps_into_seqspace}
 Let $E$ be a normalized symmetric Banach sequence space with order continuous norm or $\ell^\infty$, let $G$ be a group and let $\rho \colon G \to \Auts(E)$ be a positive representation with compact image. Then there exist a unique isometric positive representation $\pi \colon G \to S(\N)$ and an $m \in \mpos$ such that
\[ \rho_s = m \pi_s m^{-1}, \quad \forall s \in G. \]
The image of $\pi$ is compact. Conversely, any positive representation $\pi \colon G \to S(\N)$ with compact image and $m \in \mpos$ define a positive representation $\rho$ with compact image by the above equation. In this correspondence between $\rho$ and $\pi$, $\rho$ is strongly continuous if and only if $\pi$ is strongly continuous.

Moreover, if $\rho^1 = m_1 \pi^1 m_1^{-1}$ and $\rho^2 = m_2 \pi^2 m_2^{-1}$ are two positive representations with compact image, where $\pi^1$ and $\pi^2$ are isometric positive representations with compact image and $m^1, m^2 \in \mpos$, then $\rho^1$ and $\rho^2$ are order equivalent if and only if $\pi^1$ and $\pi^2$ are isometrically order equivalent.
\end{theorem}

\begin{corol}\label{c:connected_rep_seq_space}
Let $E$ be a normalized symmetric Banach sequence space with order continuous norm or $\ell^\infty$, and let $G$ be a connected compact group and let $\rho \colon G \to \Auts(E)$ be a strongly continuous positive representation. Then $\rho_s = I$ for all $s \in G$.
\end{corol}

\begin{proof}
 We know from Theorem~\ref{t:charac_reps_into_seqspace} that $\rho = m \pi m^{-1}$ for some strongly continuous isometric positive representation $\pi \colon G \to S(\N)$ and some $m \in \mpos$. For each $n \in \N$, by strong continuity the orbits $\{ \pi_s e_n: s \in G\} \subset \{e_m: m \in \N\}$ are connected. Since for $n \not= k$, $\norm{e_n - e_k} \geq 1$, the set $\{e_n: n \in \N \}$ is discrete. Hence the orbits consist of one point and $\pi$ is trivial. But then $\rho = m \pi m^{-1}$ is trivial as well.
\end{proof}

\begin{corol}\label{c:norm_discrete_seq_space}
 Let $E$ be a normalized symmetric Banach sequence space with order continuous norm or $\ell^\infty$, and let $G \subset \Auts(E)$ be a compact group. Then there exists a $\delta > 0$ such that, if $T,S \in G$, $T \not= S$ implies $\norm{T - S} \geq \delta$.
\end{corol}

\begin{proof}
 If $\phi \not= \psi \in S(\N)$, then $\norm{\phi - \psi} \geq 1$, and so the corollary follows from Theorem \ref{t:charac_group_in_seqspace}.
\end{proof}

In \cite[Corollary~3.10]{schaeferwolffarendt}, by studying the spectrum of lattice homomorphisms, this corollary is shown with $\delta = \sqrt{3} \sup\{ \norm{T}: T \in G\}$, for arbitrary uniformly bounded groups of positive operators in complex Banach lattices.

Now we will examine invariant structures under these strongly continuous positive representations. In a Banach lattice with order continuous norm, the collection of bands and the collection of closed ideal coincide by \cite[Corollary~2.4.4]{meyernieberg}. All bands in Banach sequence spaces are of the form $\{x \in E: x_n = 0 \quad \forall n \in \N \setminus A\}$ for some $A \subset \N$; this follows easily from the characterization of bands as disjoint complements. Clearly this collection coincides with the collection of projection bands and the collection of principal bands. We call a series $\sum_{n=1}^\infty x_n$ in a Riesz space unconditionally order convergent to $x$ if $\sum_{n=1}^\infty x_{\pi(n)}$ converges in order to $x$ for every permutation $\pi$ of $\N$.

\begin{theorem}\label{t:compsplitirreducible}
Let $E$ be a normalized symmetric Banach sequence space, let $G$ be a group and let $\rho \colon G \to \Aut(E)$ be a positive representation. Then $E$ splits into band irreducibles, in the sense that there exists an $\alpha$ with $1 \leq \alpha \leq \infty$ such that the set of invariant and band irreducible bands $\{B_n\}_{1 \leq n \leq \alpha}$ \textup{(}if $\alpha < \infty$\textup{)} or $\{B_n\}_{1 \leq n < \infty}$
\textup{(}if $\alpha = \infty$\textup{)} satisfies
\begin{equation}\label{e:orderconv}
  x = \sum_{n=1}^\alpha P_n x \quad \forall x \in E,
\end{equation}
where $P_n \colon E \to B_n$ denotes the band projection, and the series is unconditionally order convergent, hence, in the case that $E$ has order continuous norm, unconditionally convergent.

Moreover, if $\rho$ has compact image and $E$ has order continuous norm or $E = \ell^\infty$, then every invariant and band irreducible band is finite dimensional, and so $\alpha = \infty$.
\end{theorem}

\begin{proof}
We define the isometric positive representation $\pi := p \circ \rho \colon G \to S(\N)$, which has the same invariant bands as $\rho$ by Theorem~\ref{t:group_isom_invariant_bands}. It follows immediately from the above parametrization of the bands of $E$ that the irreducible bands are given by the orbits $\pi(G)e_n$ of the standard unit vectors $e_n$. In the case that $\rho(G)$ is compact and $E$ has order continuous norm or $E = \ell^\infty$, the map $p|_{\rho(G)}$ is continuous, so these orbits $p(\rho(G))e_n$ are compact in $E$, and hence consist of finitely many standard unit vectors, and so the irreducible bands are finite dimensional and there are countable infinitely many of them. The unconditional order convergence of the series \eqref{e:orderconv} follows from the fact that $|\pi(x)|_{\geq N} \downarrow 0$ as $N \to \infty$ for any permutation $\pi$ of $\N$.
\end{proof}

In the order continuous case, the series \eqref{e:orderconv} need not be absolutely convergent, which can be seen by taking the trivial group acting on a normalized symmetric Banach sequence space with order continuous norm not contained in $\ell^1$ and taking an $x \in E$ not in $\ell^1$.

\section{Positive representations in $C_0(\Omega)$}\label{s:cont_functions}

In this section we consider the space $C_0(\Omega)$, where $\Omega$ is a locally compact Hausdorff space. First we show that every lattice automorphism of $C_0(\Omega)$ is the product of a central lattice automorphism and an isometric lattice automorphism. We will also show that $C_0(\Omega)$ satisfies Assumption \ref{a:a1}, from which we obtain a characterization of compact groups and representations of positive groups with compact image, using the results from Section~\ref{s:abstract_characterization} and Section~\ref{s:abstract_representations}. As we will explain below, contrary to the sequence space case one cannot expect to find a direct sum type decomposition of an arbitrary strongly continuous positive representation into band irreducibles for general $C_0(\Omega)$. More investigation is necessary to determine whether these representations are still built up, in an appropriate alternative way, from band irreducible representations. As a preparation for such future research, we collect some results about the structure of invariant ideals, bands and projection bands, cf.\ Proposition~\ref{p:invariance_coo}.

Analogously to the sequence space case from Section~\ref{s:sequence_spaces}, we will start by showing that $\Aut(C_0(\Omega)) = \ZAut(C_0(\Omega)) \cdot \IAut(C_0(\Omega))$. Elements of $\Homeo$, the group of homeomorphisms of $\Omega$, are viewed as elements of $\AutC$ by $\phi x := x \circ \phi^{-1}$ for $x \in C_0(\Omega)$. The set of multiplication operators by continuous bounded functions is denoted by $C_b(\Omega)$, and by $\cbo$ we denote the elements $m \in C_b(\Omega)$ such that there exists a $\delta > 0$ such that $m(\omega) \geq \delta$ for all $\omega \in \Omega$. It follows from \cite[Theorem~3.2.10]{meyernieberg} that every $T \in \Aut(C_0(\Omega))$ can be written uniquely as a product of an element $m \in \cbo$ and an operator $\phi \in \Homeo$, so $T = m \phi$. Conversely, any $T$ defined in this way is a lattice automorphism. It is easy to see that $T \in Z(C_0(\Omega))$ if and only if its part in $\Homeo$ is trivial, so $\cbo$ is the group of central lattice automorphisms, and by Corollary \ref{c:pos_center_determines_center}, $Z(C_0(\Omega)) \cong C_b(\Omega)$. Obviously $\Homeo = \IAut(E)$, and so $\Aut(C_0(\Omega)) = \cbo \cdot \Homeo = \ZAut(C_0(\Omega)) \cdot \IAut(C_0(\Omega))$.

For $\phi \in \Homeo$ and $m \in \cbo$, define $\phi(m) := m \circ \phi^{-1} \in \cbo$. Then, for $\omega \in \Omega$ and $x \in C_0(\Omega)$,
\[ \phi m \phi^{-1} x(\omega) = m \phi^{-1} x(\phi^{-1}(\omega)) = m(\phi^{-1}(\omega)) \phi^{-1} x (\phi^{-1}(\omega)) = m(\phi^{-1}(\omega)) x(\omega), \]
so $\phi(m)$ equals the conjugation action of $\phi$ on $m$.

Our next goal is to show that Assumption \ref{a:a1} is satisfied, and for that we have to examine $\Homeo$. The topological structure of $\Homeo$ can be described by the following lemma, the proof of which is given by \cite[Definition~1.31]{williams}, \cite[Lemma~1.33]{williams} and \cite[Remark~1.32]{williams}.

\begin{lemma}\label{l:top_on_homeo}
The strong operator topology on $\Homeo$ equals the topology with as subbasis elements of the form
\[ U(K, K', V, V') := \{\phi \in \Homeo: \phi(K) \subset V \mbox{ and } \phi^{-1}(K') \subset V' \} \]
with $K$ and $K'$ compact and $V$ and $V'$ open. A net $(\phi_i)$ in $\Homeo$ converges to $\phi \in \Homeo$ if and only if $\omega_i \to \omega \in \Omega$ implies that $\phi_i(\omega_i) \to \phi(\omega)$ and $\phi_i^{-1}(\omega_i) \to \phi^{-1}(\omega)$.
\end{lemma}

Before we can show the validity of Assumption \ref{a:a1} for $C_0(\Omega)$, we need a small lemma.

\begin{lemma}\label{l:special_convergence}
 Let $(m_i)$ be a net in $\cbo$ and $(\phi_i)$ be a net in $\Homeo$ such that $m_i \phi_i$ converges strongly to the identity. If $\omega_i \to \omega \in \Omega$, then $\phi_i^{-1}(\omega_i) \to \omega$.
\end{lemma}

\begin{proof}
 Suppose that there exists a net $(\omega_i)$ converging to $\omega$ such that $\phi_i^{-1}(\omega_i)$ does not converge to $\omega$. By passing to a subnet we may assume that there exists an open neighborhood $U$ of $x$ such that $\phi_i^{-1}(\omega_i) \notin U$ for all $i$. Take a compact neighborhood $K$ of $x$ such that $K \subset U$, then by passing to a subnet we may assume that $\omega_i \in K$ for all $i$. By \cite[Lemma~1.41]{williams}, a version of Urysohn's Lemma, there exists a function $x \in C_c(\Omega)$ such that $x$ is identically one on $K$ and zero outside of $U$. Then
\begin{align*}
 \norm{m_i \phi_i x - x} &\geq | [m_i \phi_i x] (\omega_i) - x(\omega_i) | \\
&= |m_i(\omega) x(\phi_i^{-1}(\omega)) - x(\omega_i) | = |0-1| = 1,
\end{align*}
which contradicts the strong convergence of $m_i \phi_i$ to the identity.
\end{proof}

Recall that $p \colon \Aut(C_0(\Omega)) \to \Homeo$ denotes the map $m \phi \mapsto \phi$.

\begin{lemma}\label{l:assumption_verified_for_coo}
 Let $G$ be a compact subgroup of $\AutC$. Then the map $p|_G \colon G \to \Homeo$ is continuous.
\end{lemma}

\begin{proof}
Again by Lemma~\ref{l:septopgrp} it is enough to show continuity at the identity. So suppose that $(m_i)$ is a net in $\cbo$ and $(\phi_i)$ is a net in $\Homeo$ such that $(m_i \phi_i)$ is a net in $G$ converging strongly to the identity. By Lemma~\ref{l:group_is_top_group} the inverse map in $G$ is continuous, and $(m_i \phi_i)^{-1} = \phi_i^{-1}(m_i^{-1}) \phi_i^{-1}$ also converges to the identity. We have to show that $\phi_i$ converges to the identity in $\Homeo$. Let $\omega_i \to \omega \in \Omega$. By applying Lemma~\ref{l:special_convergence} twice we obtain $\phi_i^{-1}(\omega_i) \to \omega$ and $\phi_i(\omega_i) = (\phi_i^{-1})^{-1}(\omega_i) \to \omega$. This is precisely what we have to show by Lemma~\ref{l:top_on_homeo}.
\end{proof}

Hence Assumption \ref{a:a1} is satisfied, and once again we can apply the results of Section~\ref{s:abstract_characterization} and Section~\ref{s:abstract_representations}, in particular Theorem~\ref{t:characterization_mult_subgroups}, Theorem~\ref{t:characterization_rep_into_E} and Proposition \ref{p:rep_isomorphic_isometric_rep}, to obtain the following.

\begin{theorem}\label{t:characterization_compact_groups_in_coo}
 Let $\Omega$ be a locally compact Hausdorff space and let $G \subset \Auts(C_0(\Omega))$ be a compact group. Then there exist a unique compact group $H \subset \Homeo$ and an $m \in \cbo$ such that
\[ G = mHm^{-1}.\]
Conversely, any compact group $H \subset \Homeo$ and $m \in \cbo$ define a compact group $G \subset \Auts(C_0(\Omega))$ by the above equation.
\end{theorem}

\begin{theorem}\label{t:characterization_rep_on_coo}
 Let $\Omega$ be a locally compact Hausdorff space, $G$ a group and $\rho \colon G \to \Auts(C_0(\Omega))$ a positive representation with compact image. Then there exist a unique isometric positive representation $\pi \colon G \to \Homeo$ and an $m \in \cbo$ such that
\[ \rho_s = m \pi_s m^{-1} \quad \forall s \in G. \]
The image of $\pi$ is compact. Conversely, any positive representation $\pi \colon G \to \Homeo$ and $m \in \cbo$ define a positive representation $\rho$ with compact image by the above equation. In this correspondence between $\rho$ and $\pi$, $\rho$ is strongly continuous if and only if $\pi$ is strongly continuous.

Moreover, if $\rho^1 = m_1 \pi^1 m_1^{-1}$ and $\rho^2 = m_2 \pi^2 m_2^{-1}$ are two positive representations with compact image, where $\pi^1$ and $\pi^2$ are isometric positive representations with compact image and $m^1, m^2 \in \cbo$, then $\rho^1$ and $\rho^2$ are order equivalent if and only if $\pi^1$ and $\pi^2$ are isometrically order equivalent.
\end{theorem}

A part of this result is obtained in \cite[Example~4.1]{dyephillips}, by using group cohomology methods on group actions on the set $\Omega$.

Contrary to the sequence space case, our results do not, in general, lead to a decomposition of positive representations with compact image into band irreducibles. Indeed, if the trivial group acts on $C[0,1]$, then every band in $C[0,1]$ is invariant, but since every nonzero band properly contains another nonzero band, there are no invariant band irreducible bands. However, we can still say something about the various invariant structures of such representations, and for this we need a characterization of these structures in $C_0(\Omega)$.

\begin{lemma}\label{l:structures_in_coo}
 Let $\Omega$ be a locally compact Hausdorff space. Every closed ideal $I \subset \coo$ is of the form
\[ I = I_S = \{ f \in \coo: f(S) = 0 \} \]
for a unique closed $S \subset \Omega$. Hence $S \mapsto I_S$ is an inclusion reversing bijection between the closed subsets of $\Omega$ and the closed ideals of $C_0(\Omega)$. The ideal $I_S$ is a band if and only if $S$ is regularly closed, i.e., $S = \textup{int}(\overline{S})$, and it is a projection band if and only if $S$ is clopen.
\end{lemma}

\begin{proof}
 \cite[Proposition~2.1.9]{meyernieberg} and \cite[Corollary 2.1.10]{meyernieberg} show this statement for $\Omega$ compact, and the proof also works if $\Omega$ is locally compact.
\end{proof}

The next result may serve as in ingredient in the further study of positive representations in $\coo$. If $\pi \colon G \to \Homeo$ is a map, then a subset $S \subset \Omega$ is called $\pi$-invariant if $\pi_s(S) \subset S$ for all $s \in G$.

\begin{prop}\label{p:invariance_coo}
 Let $G$ be a group and let $\rho \colon G \to \Aut(\coo)$ be a positive representation, and define $\pi := p \circ \rho \colon G \to \Homeo$. Then the map $S \mapsto I_S$ from Lemma~\ref{l:structures_in_coo} restricts to a bijection between the $\pi$-invariant closed subsets $S \subset \Omega$ and the $\rho$-invariant closed ideals of $C_0(\Omega)$. By further restriction, this induces a bijection between the $\pi$-invariant regularly closed subsets of $\Omega$ and the $\rho$-invariant bands of $C_0(\Omega)$, and between the $\pi$-invariant clopen subsets of $\Omega$ and the $\rho$-invariant projection bands of $C_0(\Omega)$.
\end{prop}

\begin{proof}
 By Theorem~\ref{t:group_isom_invariant_bands} the invariant ideals of $\rho(G)$ are the same as the invariant ideals of $\pi(G) \subset \Homeo$, and so we may assume that $\rho = \pi$. Let $S \subset \Omega$ be a $\pi$-invariant closed subset. Then, for all $f \in I_S$, $s \in G$ and $\omega \in \Omega$, $\pi_s f(\omega) = f(\pi_s^{-1}(\omega)) = 0$, and so $I_S$ is a $\pi$-invariant closed ideal of $C_0(\Omega)$.

Conversely, let $I_S$ be a $\pi$-invariant closed ideal of $C_0(\Omega)$ for some closed $S \subset \Omega$. Let $\omega \in \Omega$, then for all $f \in I_S$ and $s \in G$, $f(\pi_s(\omega)) = \pi_s^{-1} f(\omega) = 0$, and so $\pi_s(\omega) \subset S$, hence $S$ is $\pi$-invariant. This shows the statement about closed ideals, and the statements about bands and projection bands follow immediately from Lemma~\ref{l:structures_in_coo}.
\end{proof}

\subsection*{Acknowledgements}
The authors thank Anthony Wickstead for providing Example~\ref{e:automorphisms_is_not_product}.

\bibliographystyle{amsplain}
\bibliography{compactgroups}

\end{document}